\documentclass[12pt]{article}
\newcommand{\reals}{{\bf R}}

\usepackage[letterpaper]{geometry}
\usepackage{graphics}

\usepackage{datetime2}

\usepackage{hyperref}

\usepackage{amsmath}

\usepackage{amsthm,amscd,amssymb}

\numberwithin{equation}{section} 
\newtheorem{theorem}[equation]{Theorem}

\begin{document}

\title{An estimate arising in scattering theory}

\author{ R.M.~Brown\footnote{ Russell Brown is partially supported by
  a grant from the Simons Foundation (\#422756).  } \\
  Department of Mathematics\\
  University of Kentucky \\
  Lexington, KY 40506-0027, USA
       }

\date{}

\maketitle

\begin{abstract}
We prove a decay estimate for an operator that arises in two-dimensional scattering problem. 
\end{abstract}

The goal of this  note is to prove an  estimate from a paper  of Klein, Sj\"ostrand and Stoilov \cite{KSS:2022} that is used to establish decay estimates for a scattering transform that arises in the study of one of the Davey-Stewartson II equations. Here, we only give the estimate and its proof and refer the reader to the manscript cited for an explanation of the significance of this estimate.

We consider a bounded domain in $ \reals ^2$  with a boundary that is locally the graph of a Lipschitz function. We define the operator $AB$ by
$$
ABu(z) = \frac 1 { 4 \pi ^ 2} \iint _{ \Omega \times \Omega} \frac {e^ { -k (\zeta - w) + \bar k ( \bar \zeta -\bar w ) } } { (z-\zeta) ( \bar \zeta - \bar w ) }
u ( w )\, dm (w) \,  dm(\zeta) . 
$$
The measure $ dm$ is Lebesgue measure on the plane. 
We let $ \langle x \rangle  = ( 1 + |x|^ 2) ^ { 1/2}$ and for $ \epsilon $ in $ \reals $, we 
let $ \langle \cdot \rangle^ \epsilon  L^p (\reals ^2)$ be the  collection of functions for which   $ \langle \cdot \rangle ^ { -\epsilon}  f \in L^p ( \reals ^2)$. Our estimate for the operator $AB$ is given in the following Theorem. 

\begin{theorem}\label{Main}
  The operator $AB$ satisfies the estimates 
\begin{align} 
  &\| ABu \|_{L^p }\leq C \| u \| _ { L^p}, \qquad && 2<p< \infty \\
  & 
  \| ABu \| _{\langle\cdot \rangle^ \epsilon L^p} \leq C\|u \|_ { L^p}, \qquad
  &&  1< p \leq 2,\   \epsilon > 2/p -1. 
\end{align}
\end{theorem}
In this theorem and below, we  use  $ L^p$ for the space $L^p(\reals^2)$ and for $L^p$ spaces on other sets or measures we use $L^p(\Omega)$ or $L^p (\mu)$. 
Before giving the proof, we recall 
three estimates for operators related to the Cauchy transform on the plane. Two of these are  well-known and one  may be less familiar.
The first is  the Hardy-Littlewood-Sobolev  estimate for the first order Riesz potential $R$  on $ \reals ^2$ \cite[p.~119]{ES:1970}. We define the operator $R$  by 
$$
Rf(x) = \int _ { \reals ^ 2} \frac { f(y)}{ |x-y|}\, dm(y) . $$ 
For this operator  and for $ \tilde p \in (1,2)$,  there is a constant $C(\tilde p)< \infty $ so that 
\begin{equation}
  \label{Riesz}
\| Rf\|_ { L^p} \leq C(\tilde p) \| f\|_ { L^{\tilde p }}, \qquad 1<\tilde p < 2,\  1/p = 1/\tilde p -1/2.
\end{equation}

Next,  we recall the Calder\'on-Zygmund inequality \cite[p.~29]{ES:1970}, specialized to the principal value singular integral  operator
$$
Sf(x) = \mbox{p.v.} \int _ { \reals ^ 2 } \frac { f(\zeta)}{(z-\zeta)^2} \, dm(\zeta). 
$$
Then for $ 1<p< \infty $, we may find a constant $C(p)$ so that 
\begin{equation}
  \label{CZ}
  \| Sf \|_ {L^p}  \leq C(p) \| f \| _ { L^p }.
\end{equation}

Finally, we recall a trace theorem for Riesz potentials due to D.R.~Adams \cite{MR287301,MR336316}.   Suppose $\mu$ is a Borel measure on $ \reals ^2$  satisfying  $ \sup _ { x, r } r^ { -1}\mu ( B(x,r) ) \leq M < \infty$. Here, $B(x,r) $ is the open ball with center $x$ and radius $r$. 
Let $Rf$ be the Riesz potential  defined above, then
$$
   \| Rf \| _ { L^ {r } ( \mu) } \leq C(\tilde p, M )  \|f\| _{ L^{ \tilde p} }.
$$
   This will hold when $ 1< \tilde p < 2$ and $ 1/r  = 2/\tilde p -1$. A convenient proof of this result may be found in  the monograph of V.G.~Maz'ja \cite[p.~54]{MR817985}.  In our  application, 
    the measure $ \mu$ is arc-length on a compact Lipschitz curve  in $ \reals ^2$ which satisfies the above hypothesis. In this case, the inequality gives an estimate for the operation of restricting the potential to a Lipschitz curve. 
   
If we view $R$ as a map from $ L^{ \tilde p} ( \reals ^ 2) $ to $ L^ r ( \mu)$, then the adjoint will map $R^* : L^s(\mu) \rightarrow L^{ s/2} ( \reals ^2)$ for $ 1<s< \infty$. Combining the results for $R$ and $R^*$,  the map $ R^* \circ R $ will satisfy
\begin{equation}
  \label{DRA}
\| R^* \circ R f \|_ { L^ p }  \leq C (\tilde p)  \| f \|_ { L^ {\tilde p } }
\end{equation}
where  $ 1<\tilde p < 2 $ and $ 1 /p = 1/\tilde p -1 /2$. 

\begin{proof}[Proof of Theorem \ref{Main}] We initially assume that $u$ is in $C_c^ \infty ( \reals ^2)$ as this makes it easier to justify the calculations below.  Once the estimate of the theorem is established, an approximation argument allows us to remove the restriction on $u$.

To begin we write 
$$
-\frac 1 k  \frac \partial { \partial \zeta } e^ { -k \zeta + \bar k \bar \zeta }  = e ^{ - k \zeta + \bar k \bar \zeta }.
$$
Let $F$ be a function in $ \Omega$ and use  Stokes theorem to obtain
$$
\int _ \Omega F( \zeta ) e^ { - k\zeta + \bar k \bar \zeta } \, d\zeta \wedge d \bar \zeta  =\frac 1 k \left (
\int_ \Omega  e^ { -k \zeta + \bar k \bar \zeta } \frac \partial { \partial \zeta}  F( \zeta ) \, d\zeta \wedge d \bar \zeta
-\int _ { \partial \Omega }  e ^ { -k \zeta + \bar k \bar \zeta }F(\zeta) \, d \bar \zeta 
\right ).
$$
 We let
$$
F( \zeta ) =  \frac 1 { z- \zeta } \int _ { \Omega } \frac {u ( w) e ^ { k w- \bar k \bar w }  } { \bar \zeta - \bar w }  \, dm (w) 
$$
and obtain the following representation for $ -8i \pi ^2 ABu$ 

\begin{align*} 
  \label{OpRep}
-8i \pi ^2 ABu(z)& = \int _ \Omega \frac { e^ { - k\zeta + \bar k \bar \zeta }}
 { z- \zeta } \int _ { \Omega } \frac {u ( w) e ^ { k w- \bar k \bar w }  } { \bar \zeta - \bar w }  \, dm(w)  
 \, d\zeta \wedge d \bar \zeta
 \\
& =\frac 1 k \left (
\mbox{p.v.} \int_ \Omega \frac {  e^ { -k \zeta + \bar k \bar \zeta } } 
 { (z- \zeta)^2 } \int _ { \Omega } \frac {u ( w) e ^ { k w- \bar k \bar w }  } { \bar \zeta - \bar w }  \, dm(w) 
\, d\zeta \wedge d \bar \zeta \right . \\
&\qquad- \pi  \int _ \Omega \frac { u (\zeta ) }{ z-\zeta}\, d\zeta \wedge d \bar \zeta 
\left. - \int _ { \partial \Omega }
\frac {e ^ { -k \zeta + \bar k \bar \zeta } }
 { z- \zeta } \int _ { \Omega } \frac {u ( w) e ^ { k w- \bar k \bar w }  } { \bar \zeta - \bar w }  \, dm(w) 
 \, d \bar \zeta \right )
 \\
 & = \frac 1 k ( I + II + III) 
\end{align*}
For the term labelled $II$ we have used that $ (\pi \bar \zeta)^ { -1}$ is fundamental solution  for $ \partial/\partial \zeta$ and thus 
$$
\frac \partial { \partial \zeta } \int _  \Omega \frac { u ( w ) e^ { kw- \bar k \bar w  }}{ \bar \zeta- \bar w } \, dm(w) = -\pi u(\zeta ) 1_ \Omega (\zeta)e^ { k \zeta - \bar k \bar \zeta }.
$$
To apply the estimates (\ref{Riesz}-\ref{DRA})  we need to recall that we have
$$ \int f( \zeta )\,  d\zeta \wedge d\bar \zeta = -2i \int f(\zeta) \, dm( \zeta) .$$

We first give the proof of Theorem \ref{Main} in the case when $ p >2$. As above, we let  $ 1/p = 1/ \tilde p - 1/2$. 
For the term $I$, we  use the estimate \eqref{Riesz} twice  and H\"older's inequality   to show that 
$$ \| I \|_{ L^ p } \leq C\|1 _ \Omega R(1_\Omega u)\| _ { L^ { \tilde p} }
\leq C\| 1_ \Omega \|_{ L^2}\| R(1_ \Omega u) \|_ { L^ p }
\leq C \| 1_ \Omega \| _ { L^2  }  \|u \| _ { L^{\tilde p}( \Omega)  }.$$
  To estimate $II$, we use \eqref{Riesz} and then \eqref{CZ} to conclude that
  $$
  \| II \| _{ L^p } \leq C \|u \| _ { L^ { \tilde p } ( \Omega) }.
  $$
  Finally, using \eqref{DRA} we have
  $$
  \|III\| _ { L^ p } \leq  C \| u \| _ { L^ {\tilde p}(\Omega) }.
  $$
If $ p>2$, then we may   use H\"older's inequality to  conclude  $ \|u\|_{L^{\tilde p}( \Omega) } \leq \| 1_ \Omega \|_{L^2}   \|u\|_{L^p}$. Thus  the last three displayed inequalities  imply the estimate of Theorem \ref{Main}.

If $ 1<p \leq 2$ and $ \epsilon > 2/p-1$, we may find $r>2$ so that $\epsilon > 2 /p -2/r > 2/p -1$ and so that $1/ \tilde r = 1/r + 1/2 \leq  1/p$.  For this $r$, we may use H\"older's inequality to see that  $ L^r \subset \langle \cdot \rangle ^ \epsilon L^p$.  We use the  three estimates above with $p$  replaced by $r$ and find
$$
\| AB u \| _ { L^r } \leq \| u \| _ { L^ {\tilde r}( \Omega)}.
$$
On the left we use that $ \|AB u \| _{ \langle \cdot \rangle ^ \epsilon L^p  } \leq C \|AB u \| _ { L^r }$ and on the right that $ \| u \| _ { L ^ { \tilde r }( \Omega) } \leq |\Omega|^ { 1/ \tilde r - 1/p} \| u \| _ { L^ p } $ to obtain our Theorem in the case $ 1 < p \leq 2$. 

\end{proof}


\def\cprime{$'$}

\small \noindent 
\today
\end{document}